\documentclass[a4paper,11pt]{article}

\usepackage[top=3.0cm, bottom=3.0cm, inner=3.0cm, outer=3.0cm, includefoot]{geometry}

\usepackage{verbatim}

\usepackage{geometry}
\usepackage{amssymb}
\usepackage{amsmath}
\usepackage{graphicx}
\usepackage{amsthm}
\usepackage{bbm}
\usepackage{subcaption}
\usepackage{enumitem}

\usepackage[T1]{fontenc}
\usepackage[utf8]{inputenc}
\usepackage{authblk}

\usepackage{tikz}
\usetikzlibrary{arrows.meta}
\usetikzlibrary{decorations.markings}

\definecolor{green}{HTML}{15B01A}

\newcommand{\IN}{{\mathbb{N}}}

\newcommand{\eChar}{\begin{enumerate}[(i)]}
\newcommand{\eCharR}{\begin{enumerate}[(a)]}
\newcommand{\eBr}{\begin{enumerate}[(1)]}

\newcommand{\diam}{\operatorname{diam}}

\newcommand{\Ind}{\operatorname{Ind}}

\newcommand\mapsfrom{\mathrel{\reflectbox{\ensuremath{\mapsto}}}}

\title
{Bonnet-Myers sharp graphs of diameter three}

\author[1]{S. Kamtue}
\affil[1]{Department of Mathematical Sciences, Durham University, UK}

\date{\today}

\theoremstyle{plain}
\newtheorem{lemma}{Lemma}[section]
\newtheorem{theorem}[lemma]{Theorem}
\newtheorem{proposition}[lemma]{Proposition}
\newtheorem{corollary}[lemma]{Corollary}

\theoremstyle{definition}

\newtheorem{definition}{Definition}[section]

\newtheorem{remark}[lemma]{Remark}

\numberwithin{equation}{section}

\begin{document}

\maketitle

\begin{abstract}
Regular graphs which are Bonnet-Myers sharp (in the sense of Ollivier Ricci curvature) and self-centered have been completely classified, and it is a natural question whether the condition of self-centeredness can be removed in the classification. We prove that this condition is indeed not necessary in the special case of Bonnet-Myers sharp graphs of diameter 3.
\end{abstract}

\section{Introduction and statement of results}

A classical theorem in Riemannian geometry is Bonnet-Myers theorem \cite{My41}, which states that a complete $n$-dimensional manifold $M$ with positive Ricci curvature lower bound $K = \inf\limits_{\substack{x\in M, v\in T_xM\\ |v|=1}} {\rm Ric}_x(v) > 0$ must be compact and its diameter has an upper bound 
\begin{equation} \label{eq:BM_RG_ineq}
{\rm diam}(M) \le \pi \sqrt{\frac{n-1}{K}}. 
\end{equation}
Moreover, Cheng's rigidity theorem \cite{Cheng75} states that equality holds in \eqref{eq:BM_RG_ineq} if and only if $M$ is the $n$-dimensional
round sphere.

In the setting of graphs, there is a discrete analogue of Bonnet-Myers theorem for Ollivier Ricci curvature (see, e.g., \cite{LLY11,Ol09}) which states that for a graph $G$ with positive curvature lower bound $K = \inf\limits_{\text{edges} \{x,y\} } \kappa(x,y) > 0$, its diameter satisfies an upper bound
\begin{equation} \label{eq:BM_Ol_ineq}
{\rm diam}(G) \le \frac{2}{K}. 
\end{equation}
where $\kappa$ is Ollivier Ricci curvature, introduced in \cite{Ol09} and modified in \cite{LLY11}. Regarding to an analogue of Cheng's rigidity result, the authors of \cite{rigidity} attempt to determine all regular graphs which attain equality in \eqref{eq:BM_Ol_ineq}. Such graphs are called \emph{Bonnet-Myers sharp}. The classification result given in \cite{rigidity}, however, is based on the additional assumption that graphs are \emph{self-centered}, that is, every vertex has another corresponding vertex which is diametrically distance apart. 

\begin{theorem} [\cite{rigidity}] \label{thm:bm_sc}
Self-centered Bonnet-Myers sharp graphs are precisely the following graphs:
\begin{enumerate}
\item hypercubes $Q^n$, $n\ge 1$;
\item cocktail party graphs $CP(n)$, $n\ge 3$;
\item the Johnson graphs $J(2n,n)$, $n\ge 3$;
\item even-dimensional demi-cubes $Q^{2n}_{(2)}$, $n\ge 3$;
\item the Gosset graph;
\end{enumerate}
and all Cartesian products $G=G_1\times G_2 \times \cdots \times G_k$ where $G_i$'s are from 1.-5. and satisfy 
$$\frac{\deg(G_1)}{\diam(G_1)}= \frac{\deg(G_2)}{\diam(G_2)}=...=  \frac{\deg(G_k)}{\diam(G_k)}.$$
\end{theorem}

It was also shown in \cite{rigidity} that the self-centeredness assumption is not necessary for Bonnet-Myers sharp graphs of diameter 2 and diameter equal to the degree. In this paper, we prove the following theorem that this assumption can also be dropped in case of diameter 3. Indeed, this result is a consequence of particular local properties of 3-balls around a pole of general Bonnet-Myers sharp graphs of arbitrary diameter (see Section \ref{sec:local} and \ref{sec:3ball}). We consider this approach as a starting point to potentially remove the self-centeredness condition for the general classification.

\begin{theorem} \label{thm:bm_diam3}
Bonnet-Myers sharp graphs of diameter 3 are self-centered, and hence are precisely the following graphs:
\begin{enumerate}
\item the cube $Q^3$;
\item the Johnson graph $J(6,3)$;
\item the demi-cube $Q^{6}_{(2)}$;
\item the Gosset graph.
\end{enumerate}
\end{theorem}

\section*{Acknowledgement}
The author would like to thank Thai Institute for the Promotion of Teaching Science and Technology for providing scholarship for his PhD study. He also express his deep gratitude to Prof. Shiping Liu from University of Science of Technology in China (USTC) for generously hosting the author's academic visit in 2019 where this project started to take place. Last but not least, the author are specially grateful to Prof. Norbert Peyerimhoff for all his helpful suggestions to the paper and moral support during the Covid-19 lockdown period. 


\section{Definitions}

\subsection{Graph theoretical notation}
All graphs $G=(V,E)$, where $V$ is the set of vertices and $E$ is the set of edges, are assumed to be finite, simple (i.e., without multiple edges and loops) and connected.
We write $x\sim y$ if there exists an edge $\{x,y\}\in E$ between two different vertices $x$ and $y$ (and we may refer to this edge as $x\sim y$). Furthermore, we use the convention $x\simeq y$ to refer to ``$x\sim y$ or $x=y$''.  For a set of vertices $A\subseteq V$, the \emph{induced subgraph} $\Ind_G(A)$ is the subgraph of $G$ whose vertex set is $A$ and whose edge set consists of all edges in $G$ that have both endpoints in $A$.
For any two vertices $x$ and $y$, the \emph{combinatorial distance} $d(x, y)$ is the length (i.e., the number of edges) in a shortest path from $x$ to $y$. Such paths of minimal length are also called \emph{geodesics} from x to y. An \emph{interval} $[x,y]$ is the set of all vertices lying on geodesics from $x$ to $y$, that is
$$[x,y] = \{v\in V: \ d(x,v) + d(v,y) = d(x,y)\}.$$
By abuse of notation, we sometimes refer to $[x,y]$ as the induced subgraphs $\Ind_G([x,y])$. The \emph{diameter} of $G$ is $\diam(G) := \max_{x,y\in V} d(x,y)$. A vertex $x\in V$ is called
a \emph{pole} if there exists another vertex $y\in V$ such that $d(x, y) = \diam(G)$, in which case $y$ will be
called an \emph{antipole} of $x$ (with respect to $G$). A graph $G$ is called \emph{self-centered} if every vertex is a pole.

For $k\in \IN$, we define the $k$-sphere of a vertex $x$ as $S_k(x)=\{v:\ d(x,v)=k\}$ and the $k$-ball of $x$ as $B_k(x) = \{v:\ d(x, v) \le k\}$. In particular, $S_1(x)$ is the set of all neighbors of $x$, and we sometimes denote this set by $N(x)$. The \emph{vertex degree} of $x$ is the number of neighbors of $x$, $\deg(x):= |N(x)|$, and a graph $G$ is called $D$-\emph{regular} if all vertices have the same degree $D$. Furthermore, let $\#_\Delta(x)$ and $\#_\Delta(x,y)$ denote the numbers of triangles containing the vertex $x$, and containing the edge $x\sim y$, respectively.

For $x,y\in V$, we also define
\begin{eqnarray*}
d^{-}_{x}(y) &:=& \left| \{ v\in N(y): \ d(x,v)=d(x,y)-1 \} \right|;\\
d^{0}_{x}(y) &:=& \left| \{ v\in N(y): \ d(x,v)=d(x,y) \} \right|; \\
d^{+}_{x}(y) &:=& \left| \{ v\in N(y): \ d(x,v)=d(x,y)+1 \} \right|,
\end{eqnarray*}
which we respectively call the in-degree, spherical-degree, and out-degree of $y$ with respect to the reference vertex $x$.

\subsection{Ollivier Ricci curvature}

Ollivier Ricci curvature was originally defined in \cite{Ol09} to characterize positive Ricci curvature by the contraction property of $L^1$-Wasserstein distance between two small balls around two different points. In this paper, we focus on the modified definition of the curvature in case of regular graphs, introduced in \cite{LLY11} and explicitly stated in \cite[Definition 2.2]{rigidity}. Let us recall the following definitions which are relevant to this curvature notion.

\begin{definition}
Let $G = (V,E)$ be a graph. Let $\mu_{1},\mu_{2}$ be two probability measures on $V$. The \emph{Wasserstein distance} $W_1(\mu_{1},\mu_{2})$ between $\mu_{1}$ and $\mu_{2}$ is defined as
\begin{equation} \label{eq:W1def} 
W_1(\mu_{1},\mu_{2}):=\inf_{\pi \in \Pi(\mu_1,\mu_2)} {\rm cost}(\pi),
\end{equation}
where $\Pi(\mu_1,\mu_2)$ is the set of all \emph{transport plans} $\pi:V\times V \rightarrow [0,1]$ satisfying the marginal constraints
\begin{equation*}
\mu_{1}(x)=\sum_{y\in V}\pi(x,y) \quad \text{and} \quad 
\mu_{2}(y)=\sum_{x\in V}\pi(x,y),
\end{equation*}
and the \emph{cost} of $\pi$ is defined as $$ {\rm cost}(\pi) := \sum_{x \in V} \sum_{y \in V} d(x,y) \pi(x,y). $$
\end{definition}

The transportation plan $\pi$ in the above definition moves a mass
distribution $\mu_1$ into a mass distribution
$\mu_2$, where $\pi(x,y)$ represent the amount of mass moved from $x$ to $y$, and $W_1(\mu_1,\mu_2)$ is a measure for the minimal effort required for such transportation. If $\pi$ attains the infimum in \eqref{eq:W1def}, we call it an {\it optimal transport plan}
from $\mu_{1}$ to $\mu_{2}$. 

\begin{definition}[\cite{rigidity}]
Let $G=(V,E)$ be a $D$-regular graph. The Ollivier Ricci curvature is defined on an edge $x\sim y$ as
$$ \kappa(x,y) = \frac{D+1}{D} \left(1 - W_1(\mu_x,\mu_y)\right),$$
where $\mu_x$ is the uniform probability distribution on the 1-ball $B_1(x)$, defined by
\begin{eqnarray*}
\mu_x(v) := 
\begin{cases}
\frac{1}{D+1} \text{ for } v \in B_1(x),\\
0 \text{ otherwise}.
\end{cases}
\end{eqnarray*}
\end{definition}

\begin{remark} \label{rem:map_exists}
Due to the fact that all individual masses in $\mu_x$ and $\mu_y$ are equal to $\frac{1}{D+1}$, there always exist an optimal transport plan $\pi$ transporting from $\mu_x$ to $\mu_y$ without splitting masses. In other words, $\pi$ is induced by a bijective \emph{optimal transport map $T: B_1(x) \rightarrow B_1(y)$}, where for each vertex $v\in B_1(x)$, $T(v)$ is the corresponding vertex in $B_1(y)$ such that $\pi(v,T(v))=\frac{1}{D+1}$ (and $\pi(v,z)=0$ for all vertices $z$ other than $T(v)$). Moreover, we can assume without loss of generality that we do not need to move masses from origin that already lie in the destination; in other words, $T(v)=v$ for all $v\in B_1(x)\cap B_1(y)$ (see arguments in, e.g., \cite[Lemma 4.1]{Idle}).
\end{remark}

A crucial theorem regarding this curvature notion is the Bonnet-Myers type diameter bound theorem, which is presented in \cite[Theorem 4.1]{LLY11} and stated as follows.

\begin{theorem}[\cite{LLY11}]
Let $G=(V,E)$ be a regular graph with a positive lower bound of Ollivier Ricci curvature $K=\inf_{x\sim y} \kappa(x,y) >0$. Then the diameter of $G$ is bounded above by 
\begin{equation} \label{eq:Ol_Bonnet-Myers}
\diam(G) \le \frac{2}{K}.
\end{equation}
\end{theorem}

Any graph satisfying \eqref{eq:Ol_Bonnet-Myers} with equality is called \emph{Bonnet-Myers sharp} graph. An equivalent definition of Bonnet-Myers sharpness can also be seen as the rigidity condition: $\inf\limits_{x\sim y} \kappa(x,y) = \frac{2}{\diam(G)}$.

\section{Relevant results about Bonnet-Myers sharp graphs}

The first result from \cite[Theorem 5.8]{rigidity} gives the relation between in-degree, spherical-degree, and out-degree, with respect to a pole of a Bonnet-Myers sharp graph.
\begin{theorem} [\cite{rigidity}]
Let $G = (V, E)$ be a $D$-regular Bonnet-Myers sharp graph with diameter $L$. Let $x_0\in V$ be a pole and $y\in S_k(x_0)$. Then
\begin{eqnarray}
d^{+}(y) - d^{-}(y) &=& D\left(1-\frac{2k}{L}\right); \label{eq:recur1} \\
2d^{+}(y) + d^{0}(y) &=& 2D\left(1-\frac{k}{L}\right); \label{eq:recur2}\\
2d^{-}(y) + d^{0}(y) &=& \frac{2kD}{L} \label{eq:recur3},
\end{eqnarray}
where $d^{-},d^{0},d^{+}$ are in-degree, spherical-degree, and out-degree of $y$ with respect to $x_0$.
\end{theorem}

In particular, the above theorem implies the following result about degrees of vertices in the 1-sphere and the 2-sphere around a pole.

\begin{corollary} \label{cor:recur}
Let $G = (V, E)$ be a $D$-regular Bonnet-Myers sharp graph with diameter $L$. Let $x_0\in V$ be a pole and $x_1\in S_1(x_0)$ and $x_2\in S_2(x_0)$.
Then
\begin{eqnarray}
&d^{+}(x_1) = D-\frac{2D}{L}+1; \qquad d^{0}(x_1)= \frac{2D}{L}-2; \label{eq:recur_S1} \\
&d^{+}(x_1)- d^{+}(x_2) = \frac{1}{2}d^{0}(x_2)+1 \label{eq:recur_S12}
\end{eqnarray}
\begin{proof} [Proof of Corollary \ref{cor:recur}]
The first two formulas follow from the fact that $d^{-}(x_1)=1$. Note also that \eqref{eq:recur2} gives $d^{+}(x_2) + \frac{1}{2}d^{0}(x_2) = D-\frac{2D}{L}=d^{+}(x_1)-1$, where the last formula follows immediately.
\end{proof}
\end{corollary}

In the previous corollary, the spherical degree $d^{0}_{x_0}(x_1)=\frac{2D}{L}-2$ can also be interpreted as $\#_\Delta(x_0,x_1)$, the number of triangles containing the edge $x_0\sim x_1$. In fact, the number of triangles containing any edge is bounded below by $\frac{2D}{L}-2$ as stated in the following proposition.

\begin{proposition} \label{prop:count_tri}
Let $G = (V, E)$ be a $D$-regular Bonnet-Myers sharp graph with diameter $L$. For any edge $x\sim y$, the number of triangles containing this edge is bounded below by $\#_\Delta(x,y) \ge \frac{2D}{L}-2$.
\end{proposition}

\begin{proof}[Proof of Proposition \ref{prop:count_tri}]
For any $D$-regular graph, there is a relation between Ollivier Ricci curvature and the number of triangles containing an edge $x\sim y$ given in, e.g., \cite[Proposition 2.7]{rigidity}: $$\kappa(x,y)\le \frac{2+\#_\Delta(x,y)}{D}.$$ Combined with the rigidity assumption \eqref{eq:Ol_Bonnet-Myers} that $\inf\limits_{x\sim y} \kappa(x,y) = \frac{2}{L}$, we have 
$$\#_\Delta(x,y) \ge D\cdot \kappa(x,y)-2 \ge \frac{2D}{L}-2.$$
\end{proof}


\section{Transport geodesics}

In this section, we recall a key concept of Bonnet-Myers sharp graphs, namely \emph{transport geodesics}, introduced in \cite{rigidity}.

Consider a $D$-regular Bonnet-Myers sharp graph $G=(V,E)$ with diameter $L$ and a pole $x_0$. Let us assume a full-length reference geodesic $\ell: x_0 \sim x_1 \sim x_2 \sim \cdots \sim x_L.$ For every $1 \le j \le L$, consider an optimal transport map
$T_j: B_1(x_{j-1}) \rightarrow B_1(x_{j})$ from $\mu_{x_{j-1}}$ to $\mu_{x_{j}}$ satisfying the following properties:
\begin{enumerate} [label=(P\arabic*)]
	\item \label{P1} $T_j: B_1(x_{j-1}) \rightarrow B_1(x_{j})$ is a bijection.
	\item \label{P2} $T_j(v)=v$ if and only if $v\in B_1(x_{j-1}) \cap B_1(x_{j})$.
\end{enumerate}
Such a map $T_j$ is called a \emph{good} optimal transport map. The existence of good maps $T_j$ is explained in Remark \ref{rem:map_exists} (but they are not necessarily unique a priori). 

Moreover, define maps $T^j: B_1(x_0) \rightarrow B_1(x_j)$ as the composition
$$T^j:= T_j \circ \cdots \circ T_1 \quad \forall 1 \le j \le L,$$
and $T^0$ is the identity map on $B_1(x_0)$ by convention. These maps $T_j$ are also bijections.

An important aspect of Bonnet-Myers sharpness is stated in the following proposition, which is in the same spirit as \cite[Propositions 7.1 and 7.2]{rigidity}.

\begin{proposition} \label{prop:transport_geo}
Let $G = (V,E)$ be a $D$-regular Bonnet-Myers sharp graph with diameter $L$. Given
a full-length reference geodesic $\ell$ and maps $T_j$ and $T^j$ defined as above. Then for every $z \in B_1(x_0)$, the vertices
$x_0, z,  T^1(z), T^2(z), ..., T^L(z), x_L$ lie along a geodesic, that is,
$$L=d(x_0,x_L)= d(x_0,z) + d(z,T^1(z)) + d(T^1(z), T^2(z)) +...+ d(T^{L-1}(z),T^L(z)) + d(T^L(z),x_L).$$ 
The sequence of vertices $x_0, z,  T^1(z), T^2(z), ..., T^L(z), x_L$  (which may contain repetitions) is called a \emph{\bf transport geodesic} along geodesic $\ell$.
\end{proposition}

\begin{proof} [Proof of Proposition \ref{prop:transport_geo}]

The idea is to calculate the cost of transportation through the sequence of probability measures $\mathbbm{1}_{x_0}, \mu_{x_0}, \mu_{x_1},...,\mu_{x_L}, \mathbbm{1}_{x_L}$. We define a positive real value
\begin{eqnarray*} 
A:= W_1(\mathbbm{1}_{x_0}, \mu_{x_0})+\sum_{j=1}^{L} W_1(u_{x_{j-1}},u_{x_j}) + W_1(\mu_{x_L}, \mathbbm{1}_{x_L}),
\end{eqnarray*}
and we inspect these three terms separately.
Firstly, observe that
\begin{eqnarray} \label{eq:W1_front}
W_1(\mathbbm{1}_{x_0}, \mu_{x_0})=\frac{1}{D+1}\sum_{z\in B_1(x_0)} d(x_0,z) = \frac{D}{D+1}.
\end{eqnarray}
Next, for each $1\le j \le L$, the distance $W_1(\mu_{x_{j-1}},\mu_{x_j})$ is equal to the cost of $T_j$ (since it is an optimal map), which can be reformulated as
\begin{eqnarray} \label{eq:W1_middles}
W_1(\mu_{x_{j-1}},\mu_{x_j})= \frac{1}{D+1} \sum\limits_{v\in B_1(x_{j-1})} d(v, T_j(v))= \frac{1}{D+1} \sum\limits_{z\in B_1(x_{0})} d(T^{j-1}(z), T^{j}(z)),
\end{eqnarray}
where the latter equality is due to labelling each $v$ by $T^{j-1}(z)$ since $T^{j-1}$ is bijective.

Lastly, since $T^{L}$ is also bijective, we have
\begin{eqnarray} \label{eq:W1_end}
W_1(\mu_{x_L}, \mathbbm{1}_{x_L}) =\frac{1}{D+1}\sum_{v\in B_1(x_L)} d(v, x_L) = \frac{1}{D+1}\sum_{z\in B_1(x_0)} d(T^{L}(z), x_L)  = \frac{D}{D+1}.
\end{eqnarray}

Combining \eqref{eq:W1_front}, \eqref{eq:W1_middles}, \eqref{eq:W1_end} and applying triangle inequality yields
\begin{eqnarray} \label{eq:bound_A}
A &=& W_1(\mathbbm{1}_{x_0}, \mu_{x_0})+\sum_{j=1}^{L} W_1(u_{x_{j-1}},u_{x_j}) + W_1(\mu_{x_L}, \mathbbm{1}_{x_L}) \nonumber \\
&=& \frac{1}{D+1} \sum_{z\in B_1(x_0)} \left( d(x_0,z)+ \sum_{j=1}^L d(T^{j-1}(z),T^{j}(z))+ d(T^{L}(z),x_L)\right) \nonumber \\
&\stackrel{\Delta}{\ge}& \frac{1}{D+1} \sum_{z\in B_1(x_0)} d(x_0,x_L) = L.
\end{eqnarray}

On the other hand, we have the rigidity assumption on \eqref{eq:Ol_Bonnet-Myers} that $\inf_{x\sim y} \kappa(x,y) =\frac{2}{L}$, which implies
\begin{equation} \label{eq:W1_middles_2}
W_1(\mu_{x_{j-1}},\mu_{x_j})=1-\frac{D}{D+1}\kappa(x_{j-1}, x_{j}) \le 1-\frac{D}{D+1}\cdot\frac{2}{L}.
\end{equation}
Combining \eqref{eq:W1_front}, \eqref{eq:W1_end}, \eqref{eq:W1_middles_2} gives
\begin{eqnarray*}
A &=& W_1(\mathbbm{1}_{x_0}, \mu_{x_0})+\sum_{j=1}^{L} W_1(u_{x_{j-1}},u_{x_j}) + W_1(\mu_{x_L}, \mathbbm{1}_{x_L}) \\
&\le& \frac{D}{D+1}+ \sum_{j=1}^{L} \left( 1-\frac{D}{D+1}\cdot\frac{2}{L} \right) + \frac{D}{D+1} = L.
\end{eqnarray*}
It follows that the triangle inequality in \eqref{eq:bound_A} holds with equality, which means
$$d(x_0,z)+ \sum_{j=1}^L d(T^{j-1}(z),T^{j}(z))+ d(T^{L}(z),x_L) = d(x_0,x_L) = L$$ for all $z\in B_1(x_0)$.

\end{proof}

The above proposition has specific reinterpretations as stated in following two corollaries, which will be more conveniently used later in the paper.

\begin{corollary} \label{cor:interval}
With the same setup as in Proposition \ref{prop:transport_geo}, let $v\in B_1(x_{j-1})$ be any vertex in the domain of the map $B_1(x_{j-1}) \stackrel{T_j}{\rightarrow} B_1(x_j)$.
Then $x_0, v, T_j(v)$ must lies on a geodesic, that is, $v\in [x_0, T_j(v)]$.
\end{corollary}
\begin{proof} [Proof of Corollary \ref{cor:interval}]
This is an immediate consequence of Proposition \ref{prop:transport_geo} with $v=T^{j-1}(z)$ by recalling that $T^{j-1}$ is bijective.
\end{proof}

\begin{corollary} \label{cor:moving} 
With the same setup as in Corollary \ref{cor:interval}, assume that $v\in S_m(x_0)$ and $T_j(v)\in S_n(x_0)$ for some $m,n\in \IN\cup\{0\}$. Then 
$$j-2 \le m\le n \le j+1,$$ and $n-m=d(v,T_j(v))$. In particular, 
\begin{enumerate}
\item $n=m$ if and only if $v=T_j(v)$;
\item $n=m+1$ if and only if $v\sim T_j(v)$.
\end{enumerate} 
\end{corollary}

\begin{proof} [Proof of Corollary \ref{cor:moving}]
Since $v\in B_1(x_{j-1})$ and $T_j(v)\in B_1(x_{j})$ as they lie in the domain and codomain of the map $T_j$, it is obvious that $j-2 \le m$ and $n\le j+1$. Moreover, since $v\in [x_0, T_j(v)]$ by Corollary \ref{cor:interval}, it means $$d(v,T_j(v))=d(x_0,T_j(v))-d(x_0,v)=n-m,$$ where the statements $1.$ and $2.$ immediately follow.
\end{proof}

In words, if vertices of $G$ are partitioned by layers $S_i(x_0)$ for $0 \le i\le L$ with respect to the pole $x_0$, the above corollary asserts that the transport map $T_j$ always sends a vertex $v$ to an outer layer (unless $v$ is fixed by the map $T_j$).

\begin{remark} \label{rem:geo_ext}
Instead of assuming a full-length reference geodesic $\ell$, let us start only with a geodesic $x_0\sim x_1 \sim x_2$ where $x_0$ is a pole and $x_1\in S_1(x_0)$ and $x_2\in S_2(x_0)$. Let $x_L$ be an antipole of $x_0$, that is, $d(x_0,x_L)=L$. It is asserted in \cite[Theorem 5.5]{rigidity} that the interval $[x_0,x_L]=V$, the set of all vertices of $G$. This immediately implies the uniqueness of the antipole $x_L$. Moreover, the fact that $v_2\in [x_0,x_L]$ implies that the geodesic $x_0\sim x_1 \sim x_2$ can always be extended to a full-length geodesic $x_0\sim x_1 \sim x_2 \sim x_3\sim \cdots \sim x_L$. This geodesic extension allows us to conveniently discuss good optimal transport maps $B_1(x_0)\stackrel{T_1}{\rightarrow} B_1(x_1) \stackrel{T_2}{\rightarrow} B_1(x_2)$ without having to define a full-length reference geodesic or other maps $T_3,...,T_L$.
\end{remark}


\section{Local properties of Bonnet-Myers sharp graphs} \label{sec:local}

In this section, we investigate local properties of a $D$-regular Bonnet-Myers sharp graph $G = (V,E)$ with diameter $L$. These properties are local in the sense that they only concern the induced subgraph of $B_3(x_0)$ of a pole $x_0$. To make it easier for readers to visualize, all of these properties are illustrated by diagrams. All diagrams display the induced subgraphs of those present vertices, that is, a pair of vertices are neighbors in $G$ if and only if they are connected by an edge in the diagram. Moreover, vertices in $S_1(x_0)$, $S_2(x_0)$, and $S_3(x_0)$ are color-coded by \textcolor{red}{red}, \textcolor{green}{green} and \textcolor{blue}{blue}, respectively.

The following two lemmas give information about intervals of length two; see Figure \ref{fig:CP_lemma}.

\begin{lemma} \label{lem:CP1}
Let $x_0$ be a pole. For any $x_1\in S_1(x_0)$ and $x_2\in S_2(x_0)$ such that $x_0\sim x_1 \sim x_2$, there exists a unique $\overline{x}_1\in S_1(x_0)$ such that $x_0\sim \overline{x}_1 \sim x_2$ and $\overline{x}_1\not\simeq x_1$. In other words, $[x_0,x_2]$ is a cocktail party graph (i.e., every vertex has exactly one antipole).

Furthermore, such $\overline{x}_1$ can be realized as $$\overline{x}_1=T_1^{-1}(x_2)=T_2(x_0),$$
where $B_1(x_0) \stackrel{T_1}{\rightarrow} B_1(x_1) \stackrel{T_2}{\rightarrow}  B_1(x_2)$ are any good optimal transport maps.
\end{lemma}

\begin{lemma} \label{lem:CP2}
Let $x_0$ be a pole. For $x_1, \overline{x}_1 \in S_1(x_0)$ such that $x_1 \not\simeq \overline{x}_1$, there exists a unique $x_2\in S_2(x_0)$ such that $x_1 \sim x_2 \sim \overline{x}_1$.
\end{lemma}

\begin{figure}[h!]
\begin{subfigure}{.5\textwidth}
\centering
\begin{tikzpicture} [scale=1.5]
\draw[fill] (0,0) circle [radius=0.025]; 
\draw[fill,red] (1,0) circle [radius=0.025]; 
\draw[fill,red] (1,1) circle [radius=0.025]; 
\draw[fill,green] (2,0) circle [radius=0.025];
\draw (0,0)--(1,0)--(2,0)--(1,1)--(0,0);
\node[below] at (0,0) {$x_0$};
\node[below,red] at (1,0) {$x_1$};
\node[above] at (1,1) {$\exists! \ \textcolor{red}{\overline{x}_1}$};
\node[below,green] at (2,0) {$x_2$};
\end{tikzpicture}
\caption{Illustration of Lemma \ref{lem:CP1}}
\label{fig:CP_lemma}
\end{subfigure}
\begin{subfigure}{.5\textwidth}
\centering
\begin{tikzpicture} [scale=1.5]
\draw[fill] (0,0) circle [radius=0.025]; 
\draw[fill,red] (1,0.5) circle [radius=0.025]; 
\draw[fill,red] (1,-0.5) circle [radius=0.025]; 
\draw[fill,green] (2,0) circle [radius=0.025];

\draw (0,0)--(1,0.5)--(2,0)--(1,-0.5)--(0,0);
\node[below] at (0,0) {$x_0$};
\node[below,red] at (1,-0.5) {$x_1$};
\node[above,red] at (1,0.5) {$\overline{x}_1$};
\node[right] at (2,0) {$\exists! \ \textcolor{green}{x_2}$};
\end{tikzpicture}
\caption{Illustration of Lemma \ref{lem:CP2}}
\end{subfigure}

\caption{Two figures illustrating Lemma \ref{lem:CP1} and \ref{lem:CP2}.}
\label{fig:CP_lemma}
\end{figure}
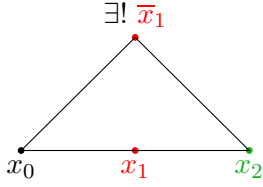
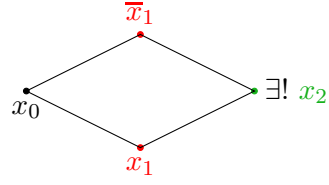

\begin{proof}[Proof of Lemma \ref{lem:CP1}]
Let $B_1(x_0)\stackrel{T_1}{\rightarrow} B_1(x_1) \stackrel{T_2}{\rightarrow} B_1(x_2)$ be any good optimal transport maps. Consider the vertex $v:=T_1^{-1}(x_2)\in B_1(x_0)$. We firstly check that $v$ is a candidate for $\overline{x}_1$. Note that $v\not=x_0$ because $T_1(x_0)=x_0$ by property \ref{P2}. Thus $v\in S_1(x_0)$. Combining with the fact that $v\in[x_0,x_2]$ due to Corollary \ref{cor:interval}, we have $x_0\sim v\sim x_2$. Furthermore, $v\not\simeq x_1$, because otherwise $v\in B_1(x_1) \cap B_1(x_0)$ would imply that $x_2=T_1(v)=v$ by property \ref{P2}, which is impossible. Hence $v=T_1^{-1}(x_2)$ is indeed a candidate for $\overline{x}_1$.
 
To prove uniqueness of such $\overline{x}_1$, assume that a vertex $v'\in S_1(x_0)$ satisfies $x_0\sim v' \sim x_2$ and $v'\not\simeq x_1$. Since $v'\in B_1(x_2)$, we can consider the vertex $u:=T_2^{-1}(v') \in B_1(x_1)$. By Corollary \ref{cor:interval}, $u\in [x_0,v']$, so either $u=x_0$ or $u=v'$. Note that $u\not=v'$ (because $u\simeq x_1$ but $v'\not\simeq x_1$). Therefore, $u=x_0$, which means $v'=T_2(u)=T_2(x_0)$ is uniquely determined as the image of $x_0$ under the bijection $T_2$. 

We can also conclude from  arguments in both paragraphs above that $$\overline{x}_1=T_1^{-1}(x_2)=T_2(x_0).$$
\end{proof}

\begin{proof}[Proof of Lemma \ref{lem:CP2}]
Let $B_1(x_0)\stackrel{T_1}{\rightarrow} B_1(x_1)$ be a good optimal transport map. Consider the vertex $T_1(\overline{x}_1) \in B_1(x_1)$. 
Firstly, note that $T_1(\overline{x}_1)\not=\overline{x}_1$ due to property \ref{P2} since $\overline{x}_1\not\in B_1(x_0)\cap B_1(x_1)$. Corollary \ref{cor:moving} (with $m=j=1$) then implies that $T_1(\overline{x}_1) \in S_2(x_0)$ and that $T_1(\overline{x}_1) \sim \overline{x}_1$. Now we know that $T_1(\overline{x}_1) \in B_1(x_1)\cap S_2(x_0) \subseteq S_1(x_1)$, which means $T_1(\overline{x}_1) \sim x_1$. Therefore, $x_1 \sim T_1(\overline{x}_1) \sim \overline{x}_1$, that is, $T_1(\overline{x}_1)$ is a candidate for $x_2$.

On the other hand, assume $v\in S_2(x_0)$ such that $x_1 \sim v \sim \overline{x}_1$. The preimage $T_1^{-1}(v)$ must coincide with $\overline{x}_1$ due to Lemma \ref{lem:CP1}. Thus $v=T_1(\overline{x}_1)$. This proves the uniqueness of $x_2$.
\end{proof}

\begin{lemma} \label{lem:XOR}
Let $x_0,x_1,x_2,\overline{x}_1$ be defined as in Lemma \ref{lem:CP1}. Let $y\in S_2(x_0)$ such that $y\sim x_2$. Then $y\sim x_1$ if and only if $y\not\sim \overline{x}_1$.
\end{lemma}

\begin{figure}[h!]
\centering
\begin{tikzpicture} [scale=1.2]

\draw[fill] (0,0) circle [radius=0.025]; 
\draw[fill,red] (1,0.5) circle [radius=0.025];
\draw[fill,red] (1,-0.5) circle [radius=0.025];
\draw[fill,green] (2,0) circle [radius=0.025];
\draw[fill,green] (2,1) circle [radius=0.025];
\draw (2,1)--(2,0)--(1,0.5)--(0,0)--(1,-0.5)--(2,0);
\draw[darkgray,dashed] (1,0.5)--(2,1)--(1,-0.5);
\node[below] at (0,0) {$x_0$};
\node[below,red] at (1,-0.5) {$x_1$};
\node[above,red] at (1,0.5) {$\overline{x}_1$};
\node[below,green] at (2,0) {$x_2$};
\node[right,green] at (2,1) {$y$};

\draw[fill] (0+5,0+1) circle [radius=0.025]; 
\draw[fill,red] (1+5,0.5+1) circle [radius=0.025];
\draw[fill,red] (1+5,-0.5+1) circle [radius=0.025];
\draw[fill,green] (2+5,0+1) circle [radius=0.025];
\draw[fill,green] (2+5,1+1) circle [radius=0.025];
\draw (2+5,1+1)--(2+5,0+1)--(1+5,0.5+1)--(0+5,0+1)--(1+5,-0.5+1)--(2+5,0+1);
\draw (1+5,0.5+1)--(2+5,1+1);
\node[below] at (0+5,0+1) {$x_0$};
\node[below,red] at (1+5,-0.5+1) {$x_1$};
\node[above,red] at (1+5,0.5+1) {$\overline{x}_1$};
\node[below,green] at (2+5,0+1) {$x_2$};
\node[right,green] at (2+5,1+1) {$y$};

\draw[fill] (0+5,0-1.5) circle [radius=0.025]; 
\draw[fill,red] (1+5,0.5-1.5) circle [radius=0.025];
\draw[fill,red] (1+5,-0.5-1.5) circle [radius=0.025];
\draw[fill,green] (2+5,0-1.5) circle [radius=0.025];
\draw[fill,green] (2+5,1-1.5) circle [radius=0.025];
\draw (2+5,1-1.5)--(2+5,0-1.5)--(1+5,0.5-1.5)--(0+5,0-1.5)--(1+5,-0.5-1.5)--(2+5,0-1.5);
\draw (2+5,1-1.5)--(1+5,-0.5-1.5);
\node[below] at (0+5,0-1.5) {$x_0$};
\node[below,red] at (1+5,-0.5-1.5) {$x_1$};
\node[above,red] at (1+5,0.5-1.5) {$\overline{x}_1$};
\node[below,green] at (2+5,0-1.5) {$x_2$};
\node[right,green] at (2+5,1-1.5) {$y$};

\draw[ultra thick, ->] (2.5,0)--(4.5,1);
\draw[ultra thick, ->] (2.5,-0.2)--(4.5,-1.5);
\node at (3.5,0) {XOR};
\end{tikzpicture}
\caption{Figure illustrating Lemma \ref{lem:XOR}. Dashed lines represent possible edges which we make no assumption if they exist or not.}
\label{fig:XOR_lemma}
\end{figure}
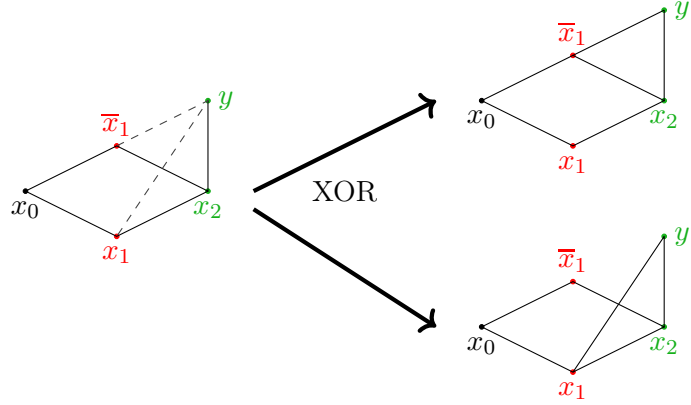

\begin{proof}[Proof of Lemma \ref{lem:XOR}]
To show the forward implication, assume for the sake of contradiction that $y\sim x_1$ and $y\sim \overline{x}_1$. Then $y$ has the same properties as $x_2$ in Lemma \ref{lem:CP2}, which contradicts to the uniqueness of such $x_2$.

To prove the backward implication, we need some counting arguments. Consider a good optimal map $B_1(x_1)\stackrel{T_2}{\rightarrow} B_1(x_2)$, and define the following sets:
\begin{eqnarray*}
A_1 &:=& \{v\in S_2(x_0): \ v\sim x_1\};\\
A_2 &:=& \{v\in S_2(x_0): \ v\sim x_1 \wedge v\sim x_2\};\\
A_3 &:=& \{v\in S_2(x_0): \ v\sim x_1 \wedge v\not\simeq x_2\};\\
A_4 &:=& \{v\in S_3(x_0): \ v\sim x_2\};\\
A_5 &:=& \{v\in S_2(x_0): \ v\sim x_2\}
\end{eqnarray*}

Firstly, note that 
\begin{equation} \label{eq:A1A2A3}
d^+(x_1)=|A_1|=1+|A_2|+|A_3|
\end{equation}
as they contribute to the out-degree of $x_1$.

For each $v\in A_3$, we know $v\not\in B_1(x_2)$, so the image $T_2(v)\not=v$ by property \ref{P2}. Corollary \ref{cor:moving} (with $m=j=2$) then implies that $T_2(v)\in S_3(x_0)$. Moreover, $T_2(v)\in B_1(x_2)$ as the codomain of $T_2$, which means $T_2(v)$ must be in $S_3(x_0)\cap B_1(x_2)= A_4$. Therefore, $T_2(A_3)\subseteq A_4$. Since $T_2$ is bijection, we have
\begin{equation} \label{eq:compare_A3A4}
|A_3|=|T_2(A_3)|\le |A_4|=d^+(x_2),
\end{equation}
which is the out-degree of $x_2$. Combining formulas \eqref{eq:A1A2A3} and \eqref{eq:compare_A3A4} gives  $$|A_2|\ge d^+(x_1)-d^+(x_2)-1=\frac{1}{2}d^0(x_2),$$
where the last equality comes from \eqref{eq:recur_S12}.

Similarly (by interchanging the roles between $x_1$ and $\overline{x}_1$), the cardinality of the set 
\begin{eqnarray*}
\overline{A}_2 &:=& \{v\in S_2(x_0): \ v\sim \overline{x}_1 \wedge v\sim x_2\}
\end{eqnarray*}
also satisfies $|\overline{A}_2|\ge \frac{1}{2}d^0(x_2)$. The forward implication shown earlier implies that $A_2\cap \overline{A}_2=\emptyset$. Therefore,
$$|A_2 \cup \overline{A}_2|=|A_2|+|\overline{A}_2|\ge d^0(x_2).$$

On the other hand, both $A_2$ and $\overline{A}_2$ are subsets of $A_5$, so $$|A_2 \cup \overline{A}_2| \le |A_5|=d^0(x_2).$$
Therefore, all the inequalities above indeed hold with equality. Consequently, $A_5$ is partitioned into $A_2\sqcup \overline{A}_2$. In other words, for every $y\in S_2(x_0)$ such that $y\sim x_2$, either $y\sim x_1$ or $y\sim \overline{x}_1$ as desired.

\end{proof}

\begin{lemma} \label{lem:prop_T2}
Let $x_0,x_1,x_2$ be defined as in Lemma \ref{lem:CP1} and let $B_1(x_1) \stackrel{T_2}{\rightarrow} B_1(x_2)$ be a good optimal transport map. Let $v\in B_1(x_1)$ and $w\in B_1(x_2)$ such that $T_2(v)=w$. Then the following statements are true.
\begin{itemize}
\item[(a)] If $w\in S_3(x_0)$, then $v\in S_2(x_0)$ and $v\not\simeq x_2$ and $v\sim w$.
\item[(b)] It is impossible that $v\in S_1(x)$ and $w\in S_3(x_0)$ simultaneously.
\item[(c)] If $v\in S_1(x_0)$ and $v\not\sim x_2$, then $w\in S_2(x_0)$ and $w\not\sim x_1$ and $v\sim w \sim x_2$.
\end{itemize}
See Figures \ref{fig:T2_map_a}, \ref{fig:T2_map_b}, \ref{fig:T2_map_c} for respective cases. 
\end{lemma}

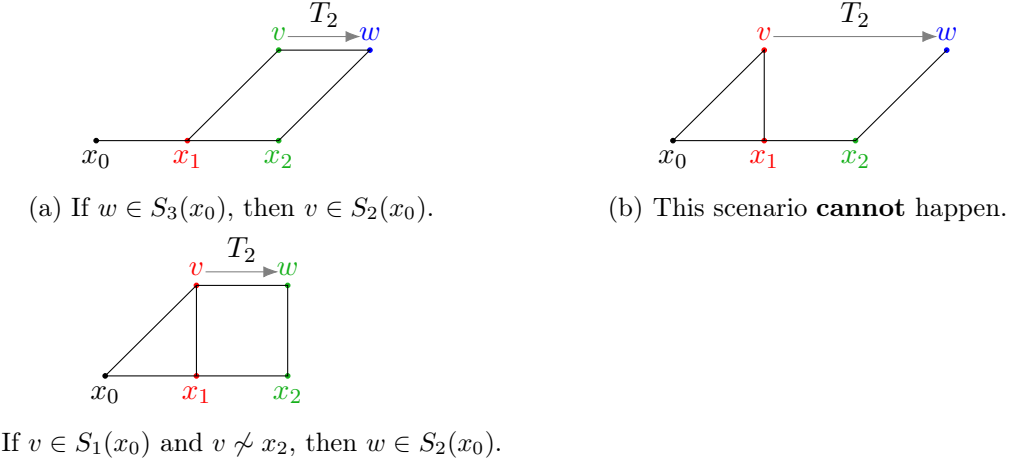
\begin{figure}[h!]
\begin{subfigure}{.6\textwidth}
\centering
\begin{tikzpicture} [scale=1.2]
\draw[fill] (0,0) circle [radius=0.025]; 
\draw[fill,red] (1,0) circle [radius=0.025]; 
\draw[fill,green] (2,0) circle [radius=0.025];
\draw[fill,green] (2,1) circle [radius=0.025]; 
\draw[fill,blue] (3,1) circle [radius=0.025]; 
\draw (0,0)--(1,0)--(2,0)--(3,1)--(2,1)--(1,0);
\draw[-{Latex[length=2mm]},gray] (2.1,1.15)--(2.9,1.15);
\node[below] at (0,0) {$x_0$};
\node[below,red] at (1,0) {$x_1$};
\node[below,green] at (2,0) {$x_2$};
\node[above,green] at (2,1) {$v$};
\node[above,blue] at (3,1) {$w$};
\node[above] at (2.5,1.15) {$T_2$};
\end{tikzpicture}
\caption{If $w\in S_3(x_0)$, then $v\in S_2(x_0)$.}
\label{fig:T2_map_a}
\end{subfigure}
\begin{subfigure}{.4\textwidth}
\centering
\begin{tikzpicture} [scale=1.2]
\draw[fill] (0,0) circle [radius=0.025]; 
\draw[fill,red] (1,0) circle [radius=0.025]; 
\draw[fill,red] (1,1) circle [radius=0.025]; 
\draw[fill,green] (2,0) circle [radius=0.025];
\draw[fill,blue] (3,1) circle [radius=0.025]; 
\draw (1,0)--(1,1)--(0,0)--(1,0)--(2,0)--(3,1);
\draw[-{Latex[length=2mm]},gray] (1.1,1.15)--(2.9,1.15);
\node[below] at (0,0) {$x_0$};
\node[below,red] at (1,0) {$x_1$};
\node[above,red] at (1,1) {$v$};
\node[below,green] at (2,0) {$x_2$};
\node[above,blue] at (3,1) {$w$};
\node[above] at (2,1.15) {$T_2$};
\end{tikzpicture}
\caption{This scenario \textbf{cannot} happen.}
\label{fig:T2_map_b}
\end{subfigure}
\begin{subfigure}{.6\textwidth}
\centering
\begin{tikzpicture} [scale=1.2]
\draw[fill] (0,0) circle [radius=0.025]; 
\draw[fill,red] (1,0) circle [radius=0.025]; 
\draw[fill,red] (1,1) circle [radius=0.025]; 
\draw[fill,green] (2,0) circle [radius=0.025];
\draw[fill,green] (2,1) circle [radius=0.025];
\draw[fill,white] (3,0) circle [radius=0.025]; 
\draw (0,0)--(1,0)--(2,0)--(2,1)--(1,1)--(0,0);
\draw (1,0)--(1,1);
\draw[-{Latex[length=2mm]},gray] (1.1,1.15)--(1.9,1.15);
\node[below] at (0,0) {$x_0$};
\node[below,red] at (1,0) {$x_1$};
\node[above,red] at (1,1) {$v$};
\node[below,green] at (2,0) {$x_2$};
\node[above,green] at (2,1) {$w$};
\node[above] at (1.5,1.15) {$T_2$};
\end{tikzpicture}
\caption{If $v\in S_1(x_0)$ and $v\not\sim x_2$, then $w\in S_2(x_0)$.}
\label{fig:T2_map_c}
\end{subfigure}

\caption{Figures illustrating map $B_1(x_1) \stackrel{T_2}{\rightarrow} B_1(x_2)$ explained in Lemma \ref{lem:prop_T2}.}
\label{fig:T2_map}
\end{figure}

\begin{proof}[Proof of Lemma \ref{lem:prop_T2}]
Recall from the proof of Lemma \ref{lem:XOR} that the formula \eqref{eq:compare_A3A4} holds with equality, so we have $T_2(A_3)= A_4$. It means if $w\in A_4=B_1(x_2)\cap S_3(x_0)$, then $w=T_2(v)$ for some $v\in A_3$, that is $v\in S_2(x_0)$ and $v\not\simeq x_2$. Moreover, $v\sim w$ is due to Corollary \ref{cor:moving} (with $m=2$ and $n=3$). Thus the statement (a) is true. 

The statement (b) is simply a reinterpretation of (a). For the statement (c), assume that $v\in S_1(x_0)$ and $v\not\sim x_2$. Since $v\not\in B_1(x_1)\cap B_1(x_2)$, we have $w=T_2(v)\not=v$ due to property \ref{P2}. Corollary \ref{cor:moving} (with $m=1$ and $j=2$) further implies that $w\in S_2(x_0)$ or $w\in S_3(x_0)$. However, the latter case cannot happen due to (b), so $w\in S_2(x_0)$. Moreover, $w\not\sim x_1$; otherwise, $w\in B_1(x_1)\cap B_1(x_2)$ would then imply by property \ref{P2} that $T_2(v)=w=T_2(w)$, which is impossible since $w\not=v$ and $T_2$ is bijective. Recall also that $w\in B_1(x_2)$ but $w\not=x_2$ (since $w\not\sim x_1$), so $w\sim x_2$. Finally, the fact that $v\sim w$ follows from Corollary \ref{cor:moving} (with $m=1$ and $n=2$). Statement (c) is therefore proved.
\end{proof}

\begin{lemma} \label{lem:CP_tight}
Let $x_0,x_1,x_2,\overline{x}_1$ be defined as in Lemma \ref{lem:CP1} and let $u \in S_1(x_0)$ and $u\not\in \{x_1,\overline{x}_1\}$. Then $u\in [x_0,x_2]$ if and only if $u$ is connected to both $x_1$ and $\overline{x}_1$.
\end{lemma}

\begin{figure}[h!]
\centering
\begin{tikzpicture} [scale=1.5]
\draw[fill] (0,0) circle [radius=0.025]; 
\draw[fill,red] (1,1) circle [radius=0.025];
\draw[fill,red] (1,-1) circle [radius=0.025]; 
\draw[fill,red] (0.9,0) circle [radius=0.025]; 
\draw[fill,green] (2,0) circle [radius=0.025];
\draw (0,0)--(1,1)--(2,0)--(1,-1)--(0,0);
\node[below] at (0,0) {$x_0$};
\node[below,red] at (1,-1) {$x_1$};
\node[below,red] at (0.9,0) {$u$};
\node[above,red] at (1,1) {$\overline{x}_1$};
\node[below,green] at (2,0) {$x_2$};
\draw (0,0)--(2,0);
\draw[darkgray,dashed] (1,1)--(0.9,0);
\draw[darkgray,dashed] (0.9,0)--(1,-1);

\draw[fill] (0+5,0) circle [radius=0.025]; 
\draw[fill,red] (1+5,1) circle [radius=0.025];
\draw[fill,red] (1+5,-1) circle [radius=0.025]; 
\draw[fill,red] (0.9+5,0) circle [radius=0.025]; 
\draw[fill,green] (2+5,0) circle [radius=0.025];
\draw (0+5,0)--(1+5,1)--(2+5,0)--(1+5,-1)--(0+5,0);
\node[below] at (0+5,0) {$x_0$};
\node[below,red] at (1+5,-1) {$x_1$};
\node[below,red] at (0.9+5,0) {$u$};
\node[above,red] at (1+5,1) {$\overline{x}_1$};
\node[below,green] at (2+5,0) {$x_2$};
\draw (0+5,0)--(0.9+5,0);
\draw (1+5,1)--(0.9+5,0)--(1+5,-1);
\draw[darkgray,dashed] (0.9+5,0)--(2+5,0);
\draw[ultra thick, <->] (3,0)--(4,0);
\end{tikzpicture}
\caption{Figure illustrating Lemma \ref{lem:CP_tight}.}
\label{fig:CP_tight_lemma}
\end{figure}

\begin{proof}[Proof of Lemma \ref{lem:CP_tight}]
The forward implication is due to the interval $[x_0,x_2]$ being a cocktail party graph by Lemma \ref{lem:CP1}.

For backward implication, assume $u\in S_1(x_0)$ such that  $x_1\sim u\sim \overline{x}_1$, and we would like to show that $u\sim x_2$.
We will go a long way around for some counting arguments.

Fix an arbitrary vertex $a\in S_1(x_0)$. We define the following sets:
\begin{eqnarray*}
A_1 &:=& \{ (b,c)\in S_1(x_0)^2: \ a\sim b\sim c \wedge a\not=c\}; \\
A_2 &:=& \{ (b,c)\in S_1(x_0)^2: \ a\sim b\sim c  \wedge a\not\simeq c \};\\
A_3 &:=& \{ (b,c)\in S_1(x_0)^2: \ a\sim b\sim c  \wedge a\sim c \};\\
A_4 &:=& \{(b,c')\in S_1(x_0)\times S_2(x_0): \ a\sim b\sim c'  \wedge a\sim c' \};\\
A_5 &:=& \{(b,c)\in S_1(x_0)\times V: \  a\sim b\sim c \sim a \}.
\end{eqnarray*}
To count $A_1$, there are $t$ possibilities for vertex $b$, and for each $b$ there are $t-1$ possible vertices for $c$, where $t:=d^0(v)=\frac{2D}{L}-2$ for any $v\in S_1(x_0)$; see formula \eqref{eq:recur_S1}. Thus $|A_1|=t(t-1)$. Note also that $A_1$ is partitioned into $A_2 \sqcup A_3$.

Next, we would like to compare $A_2$ and $A_4$. We firstly realize a one-to-one correspondence between
$$C_1:=\{c\in S_1(x_0): \ c\not\simeq a\} \qquad \text{and} \qquad C_2:=\{c'\in S_2(x_0): \ c'\sim a\}$$ as follows.
For each $c\in C_1$, there is a unique $c'\in C_2$ with $c'\sim c$ by Lemma \ref{lem:CP2}. On the other hand, for any $c'\in C_2$, there is a unique $c\in C_1$ with $c\sim c'$ by Lemma \ref{lem:CP1}. Note that each corresponding pair $(c,c')$ forms a quadrilateral $\square x_0ac'c$. For each such pair $(c,c')$, all possible vertices $b\in S_1(x_0)$ that $a\sim b\sim c'$ are exactly all those vertices inside the interval $[x_0,c']$ excluding $\{x_0,a,c',c\}$ since $[x_0,c']$ is a cocktail party graph by Lemma \ref{lem:CP1}. At the same time, these vertices $b$ cannot be an antipole of $c$ in the cocktail party graph $[x_0,c']$ (because the only antipole of $c$ is $a$), so we have $a\sim b \sim c$. The above argument means 
\begin{equation} \label{eq:compare_A2A4}
(b,c')\in A_4 \quad \Leftrightarrow \quad b\in [x_0,c']\backslash\{x_0,a,c',c\} \quad \Rightarrow \quad (b,c)\in A_2.
\end{equation}
Note also that the map $(b,c)\mapsto (b,c')$ induced by \eqref{eq:compare_A2A4} is injective, which results in $|A_4|\le |A_2|$.

We wish to prove the converse of \eqref{eq:compare_A2A4}, that is, to prove that the map $(b,c)\mapsto (b,c')$ is surjective. The desired result $u\sim x_2$ from the beginning of the proof will then follow by substituting $a=x_1$, $c=\overline{x}_1$, $c'=x_2$ and $b=u$. To verify this surjection, it suffices to show that $|A_4|= |A_2|$.

For the cardinality of $A_5$, we need to count the number of triangles $\triangle abc$ such that $b\in S_1(x_0)$ and $a\sim b$. A vertex $c$ could either come from $S_1(x_0)$ or $S_2(x_0)$, or $c=x_0$. In case $c\in S_1(x_0)$, the number of $\triangle abc$ is exactly $|A_3|$. In case $c\in S_2(x_0)$, the number of $\triangle abc$ is exactly $|A_4|$. If $c=x_0$, then there are $t$ possibilities for such a vertex $b$. Overall, we have 
\begin{equation} \label{eq:count_A5}
|A_5|=|A_3|+|A_4|+t \le |A_3|+|A_2|+t=|A_1|+t=t(t-1)+t=t^2.
\end{equation}

On the other hand, since there are $d^0(a)=t$ possibilities for $b\in S_1(x_0)$ such that $b\sim a$, and for each $b$, we know that $\#_\Delta(a,b)\ge t$ from Proposition \ref{prop:count_tri}. It means $|A_5| \ge t^2$, which implies that the inequality \eqref{eq:count_A5} holds equality. Therefore $|A_4|=|A_2|$ as desired.
\end{proof}

The following two lemmas, which are inverse of each other, explain the situation when switching the roles between $x_1$ and $\overline{x}_1$; see Figure \ref{fig:CP_swap_lemma}.

\begin{lemma} \label{lem:CP_swap1}
Let $x_0,x_1,\overline{x}_1,x_2$ be as defined in Lemma \ref{lem:CP1}. Let $u\in S_1(x_0)$ such that $u\sim x_1$ but $u\not\sim x_2$, and let $v\in S_2(x_0)$ such that $u\sim v \sim x_2$ but $v\not\sim x_1$. Then $x_0,u,v,\overline{x}_1$ form a quadrilateral with $u\not\sim \overline{x}_1$. In fact, for any such vertex $u$, there exists a unique such vertex $v$, namely $v=T_2(u)$ where $B_1(x_1) \stackrel{T_2}{\rightarrow} B_1(x_2)$ is a good optimal transport map.
\end{lemma}

\begin{lemma} \label{lem:CP_swap2}
Let $x_0,x_1,\overline{x}_1,x_2$ be as defined in Lemma \ref{lem:CP1}. Let $u\in S_1(x_0)$ such that $u\not\sim \overline{x}_1$ and $u\not\sim x_2$, and let $v\in S_2(x_0)$ such that $u\sim v \sim x_2$ and $v\sim \overline{x}_1$. Then $x_1,u,v,x_2$ form a quadrilateral with $u\not\sim x_2$ and $v\not\sim x_1$.
\end{lemma}

\begin{figure}[h!]
\centering
\begin{tikzpicture} [scale=1.5]
\draw[fill] (0,0) circle [radius=0.025]; 
\draw[fill,red] (1,1) circle [radius=0.025]; 
\draw[fill,red] (1,0) circle [radius=0.025]; 
\draw[fill,green] (2,0) circle [radius=0.025];
\draw[fill,green] (2,1) circle [radius=0.025];
\draw (0,0)--(1,1)--(2,1)--(2,0)--(0,0);
\draw (1,0)--(1,1);
\node[below] at (0,0) {$x_0$};
\node[below,red] at (1,0) {$x_1$};
\node[above,red] at (1,1) {$u$};
\node[below,green] at (2,0) {$x_2$};
\node[above,green] at (2,1) {$v$};

\draw[fill] (0+5,0) circle [radius=0.025]; 
\draw[fill,red] (1+5,0) circle [radius=0.025]; 
\draw[fill,red] (1+5,1) circle [radius=0.025]; 
\draw[fill,green] (2+5,0) circle [radius=0.025];
\draw[fill,green] (2+5,1) circle [radius=0.025];
\draw (0+5,0)--(1+5,1)--(2+5,1)--(2+5,0)--(0+5,0);
\draw (1+5,0)--(2+5,1);
\node[below] at (0+5,0) {$x_0$};
\node[below,red] at (1+5,0) {$\overline{x}_1$};
\node[above,red] at (1+5,1) {$u$};
\node[below,green] at (2+5,0) {$x_2$};
\node[above,green] at (2+5,1) {$v$};

\draw[ultra thick, ->] (2.5,0.6)--(4.5,0.6);
\draw[ultra thick, <-] (2.5,0.45)--(4.5,0.45);
\node[above] at (3.5,0.6) {Lemma \ref{lem:CP_swap1}};
\node[below] at (3.5,0.45) {Lemma \ref{lem:CP_swap2}};

\end{tikzpicture}
\caption{Figure illustrating Lemma \ref{lem:CP_swap1} and \ref{lem:CP_swap2}}
\label{fig:CP_swap_lemma}
\end{figure}

\begin{proof} [Proof of Lemma \ref{lem:CP_swap1}]
Firstly, note that such a vertex $v$ exists by considering $v:=T_2(u)$ and recall Lemma \ref{lem:prop_T2}(c).

Since $v\not\sim x_1$ (by assumption), it is implied by Lemma \ref{lem:XOR} that $v\sim \overline{x}_1$. Moreover, $u\sim x_1$ (by assumption) implies that $u\not\sim \overline{x}_1$ (otherwise, $u$ connected to both $x_1$ and $\overline{x}_1$ would imply by Lemma \ref{lem:CP_tight} that $u\in [x_0,x_2]$, which contradicts to the assumption that $u\not\sim x_2$). Therefore, we have a quadrilateral $\square x_0uv\overline{x}_1$ with $u\not\sim \overline{x}_1$ as desired. The uniqueness of $v$ can be seen by Lemma \ref{lem:CP2} (with respect to $x_0,u,\overline{x}_1$).
\end{proof}

\begin{proof} [Proof of Lemma \ref{lem:CP_swap2}]
Consider the vertex $u':=T_2^{-1}(v)\in B_1(x_1)$ and we claim that $u'$ and $v$ satisfy
\begin{enumerate}
\item $u'\in S_1(x_0)$ and $u'\sim x_1$ but $u'\not\sim x_2$;
\item $v\in S_2(x_0)$ and $u'\sim v \sim x_2$ but $v\not\sim x_1$.
\end{enumerate}

If we manage to prove these two claims and show that indeed $u'=u$, we will be done. 

In claim 2, the fact that $v\in S_2(x_0)$ and $v \sim x_2$ is actually a part of assumption on the vertex $v$. Moreover, the fact that $v\not\sim x_1$ follows from Lemma \ref{lem:XOR} since $v\sim \overline{x}_1$. We are left to prove in claim 2 that $u'\sim v$.

Note that $u'\in [x_0,v]$ by Corollary \ref{cor:interval}. However, $u'\not=x_0$ (since $T_2(x_0)=\overline{x}_1$ by Lemma \ref{lem:CP1}) and $u'\not=v$ (since $v\not\in B_1(x_1)$ but $u'\in B_1(x_1)$). Therefore $u'$ must lie inside the interior of $[x_0,v]$ (i.e., excluding endpoints $x_0,v$). It means that $x_0\sim u' \sim v$ and $u'\in S_1(x_0)$. This finishes the proof of claim 1. Furthermore, we had $u'\in B_1(x_1)$ but $u'\not=x_1$ (because $T_2(x_1)=x_1$ by property \ref{P2}), so $u'\sim x_1$. Lastly, $u'\not\sim x_2$; otherwise, $u'\in B_1(x_1)\cap B_1(x_2)$ would imply that $v=T_2(u')=u'$ by property \ref{P2}, which is false. This finishes the proof of claim 1.

Now since $u'$ and $v$ satisfy all claimed properties, Lemma \ref{lem:CP_swap1} (with $u$ replaced by $u'$) implies that $x_0,u',v,\overline{x}_1$ form a quadrilateral with $u'\not\sim \overline{x}_1$. Recall earlier that $x_0,u,v,\overline{x}_1$ also form a quadrilateral with $u\not\sim \overline{x}_1$ by assumption. In view of Lemma \ref{lem:CP1}, the uniqueness of the antipole of $\overline{x}_1$ in the cocktail party graph $[x_0,v]$ then implies $u'=u$ as desired.
\end{proof}


\section{Intervals of length three} \label{sec:3ball}

Continuing from the previous section, this section provides results concerning intervals of length three starting at a pole $x_0$ of a $D$-regular Bonnet-Myers sharp graph $G=(V,E)$.

Let us start with $x_0$, $x_1$, $\overline{x}_1$, $x_2$ defined as in Lemma \ref{lem:CP1}. As explained Remark \ref{rem:geo_ext}, the extension of geodesics $x_0\sim x_1\sim x_2$ and $x_0\sim \overline{x}_1 \sim x_2$ allows us to simultaneously consider good optimal transport maps $B_1(x_0) \stackrel{T_1}{\rightarrow} B_1(x_1) \stackrel{T_2}{\rightarrow}  B_1(x_2)$ and $B_1(x_0) \stackrel{\overline{T}_1}{\rightarrow} B_1(\overline{x}_1) \stackrel{\overline{T}_2}{\rightarrow}  B_1(x_2)$.

\begin{theorem} \label{thm:uvw}
Let $u\in S_1(x_0)$. Assume that $u\stackrel{T_1}{\mapsto} v \stackrel{T_2}{\mapsto} w$, and that $u,v,w$ are all different. Then $w\in S_3(x_0)$, and $u$ and $x_2$ are antipoles of each other with respect to the interval $[x_0,w]$. In other words, $u,x_2\in [x_0,w]$ and $d(u,x_2)$=3.
\end{theorem}

\begin{proof} [Proof of theorem \ref{thm:uvw}]
Given that $u\in S_1(x_0)$ and $u,v,w$ are all different, the following fact can be seen from Corollary \ref{cor:moving}.

\underline{\bf Fact 1} $u\in S_1(x_0)$, $v\in S_2(x_0)$, $w\in S_3(x_0)$, and $u \sim v \sim w$.

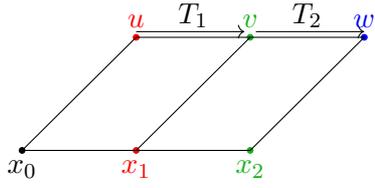
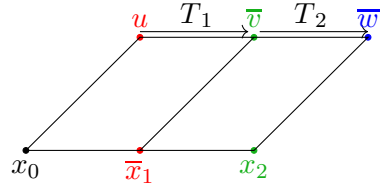
\begin{figure}[h!]
\begin{subfigure}{0.5\textwidth}
\centering
\begin{tikzpicture} [scale=1.5]
\draw[fill] (0,0) circle [radius=0.025]; 
\draw[fill,red] (1,0) circle [radius=0.025]; 
\draw[fill,red] (1,1) circle [radius=0.025]; 
\draw[fill,green] (2,0) circle [radius=0.025];
\draw[fill,green] (2,1) circle [radius=0.025]; 
\draw[fill,blue] (3,1) circle [radius=0.025]; 
\draw (0,0)--(1,0)--(2,0)--(3,1)--(2,1)--(1,1)--(0,0);
\draw (1,0)--(2,1);
\draw[->] (1,1+.05)--(2-0.05,1+.05);
\draw[->] (2+0.05,1+.05)--(3,1+.05);
\node[below] at (0,0) {$x_0$};
\node[below,red] at (1,0) {$x_1$};
\node[above,red] at (1,1) {$u$};
\node[below,green] at (2,0) {$x_2$};
\node[above,green] at (2,1) {$v$};
\node[above,blue] at (3,1) {$w$};
\node[above] at (1.5,1) {$T_1$};
\node[above] at (2.5,1) {$T_2$};
\end{tikzpicture}
\caption{Summary of Fact 1 and Claim 1 and 2}
\label{fig:uvw_1}
\end{subfigure}
\begin{subfigure}{0.5\textwidth}
\centering
\begin{tikzpicture} [scale=1.5]
\draw[fill] (0,0) circle [radius=0.025]; 
\draw[fill,red] (1,0) circle [radius=0.025]; 
\draw[fill,red] (1,1) circle [radius=0.025]; 
\draw[fill,green] (2,0) circle [radius=0.025];
\draw[fill,green] (2,1) circle [radius=0.025]; 
\draw[fill,blue] (3,1) circle [radius=0.025]; 
\draw (0,0)--(1,0)--(2,0)--(3,1)--(2,1)--(1,1)--(0,0);
\draw (1,0)--(2,1);
\draw[->] (1,1+.05)--(2-0.05,1+.05);
\draw[->] (2+0.05,1+.05)--(3,1+.05);
\node[below] at (0,0) {$x_0$};
\node[below,red] at (1,0) {$\overline{x}_1$};
\node[above,red] at (1,1) {$u$};
\node[below,green] at (2,0) {$x_2$};
\node[above,green] at (2,1) {$\overline{v}$};
\node[above,blue] at (3,1) {$\overline{w}$};
\node[above] at (1.5,1) {$\overline{T}_1$};
\node[above] at (2.5,1) {$\overline{T}_2$};
\end{tikzpicture}
\caption{Summary of Claim 7}
\label{fig:uvw_2}
\end{subfigure}
\caption{Summary of all the seven claims}
\label{fig:uvw}
\end{figure}

Let $\overline{v},\overline{w}$ be the images $u\stackrel{\overline{T}_1}{\mapsto} \overline{v} \stackrel{\overline{T}_2}{\mapsto} \overline{w}$. We would like to verify the following seven claims.\\
\underline{\bf Claim 1} $u\not=x_1$, $v\not=x_2$, $u\not=\overline{x}_1$, and $\overline{v}\not=x_2$.\\
\underline{\bf Claim 2} $u\not\sim x_1$, $u\not\sim x_2$, and $v\not\sim x_2$.\\
\underline{\bf Claim 3} $u\not\sim \overline{x}_1$.\\
\underline{\bf Claim 4} $u$ and $\overline{v}$ are different, $\overline{v}\in S_2(x_0)$ and $u\sim \overline{v}$.\\
\underline{\bf Claim 5} $\overline{v}\not\sim x_2$.\\
\underline{\bf Claim 6} $\overline{v}$ and $\overline{w}$ are different, $\overline{w}\in S_3(x_0)$, and $\overline{v} \sim \overline{w}$.\\
\underline{\bf Claim 7} $u\in S_1(x_0)$, $\overline{v}\in S_2(x_0)$, $\overline{w}\in S_3(x_0)$, and $u \sim \overline{v} \sim \overline{w}$.

Indirect proof for Claim 1:\\
If $u=x_1$, then $v=T_1(u)=u$ by property \ref{P2} of the map $T_1$. \\
If $v=x_2$, then $w=T_2(v)=v$ by property \ref{P2} of the map $T_2$. \\
If $u=\overline{x}_1$, then $v=T_1(u)=T_1(\overline{x}_1)=x_2$ by Lemma \ref{lem:CP1}. This contradicts to $v\not=x_2$, shown previously.\\
Lastly, if $\overline{v}=x_2$, then $u=\overline{T}_1^{-1}(v)=\overline{T}_1^{-1}(x_2)=x_1$ by Lemma \ref{lem:CP1} (where the roles of $x_1$ and $\overline{x}_1$ are interchanged and the map $T_1$ is replaced by $\overline{T}_1$). This contradicts to $u\not=x_1$, shown previously.

Indirect proof for Claim 2:\\
If $u\sim x_1$, then $v=T_1(u)=u$ by property \ref{P2} of the map $T_1$.\\
If $u\sim x_2$, then the fact that $u\not\sim x_1$ implies that $u=\overline{x}_1$ by the uniqueness in Lemma \ref{lem:CP1}. This contradicts to Claim 1.\\
If $v\sim x_2$, then $w=T_2(v)=v$ by property \ref{P2} of the map $T_2$.

Fact 1 and Claim 1 and 2 altogether can be illustrated by Figure \ref{fig:uvw_1}.

Next, we will prove Claim 3 also by indirect proof. Assume for the sake of contradiction that $u\simeq \overline{x}_1$, i.e., $u\in B_1(\overline{x}_1)$. The fact that $u\not=\overline{x}_1$ from Claim 1 means $u\sim \overline{x}_1$. Since $u\in B_1(\overline{x}_1)$ lies in the domain of $\overline{T}_2$, we can define a vertex $v':=\overline{T}_2(u)$. Now we have $u\in B_1(\overline{x}_1) \cap S_1(x_0)$ and $u\not \sim x_2$ (from Claim 2). We can then apply Lemma \ref{lem:prop_T2}(c) (with the map $T_2$ replaced by $\overline{T}_2$ and with vertices $x_1,v,w$ substituted by $\overline{x}_1,u,v'$, respectively) to conclude that $v'\in S_2(x_0)$ and $v'\not\sim \overline{x}_1$ and $u\sim v'\sim x_2$. To visualize this, you may compare Figure \ref{fig:T2_map_c} to Figure \ref{fig:claim_uv} (left diagram) with the mentioned substitution.
 
Lemma \ref{lem:CP_swap1} then implies that $v'\sim x_1$ and $v'\sim x_2$, as illustrated in Figure \ref{fig:claim_uv} (which can be compared to Figure \ref{fig:CP_swap_lemma} where the roles of $x_1$ and $\overline{x}_1$ are interchanged, and $v$ is replaced by $v'$).

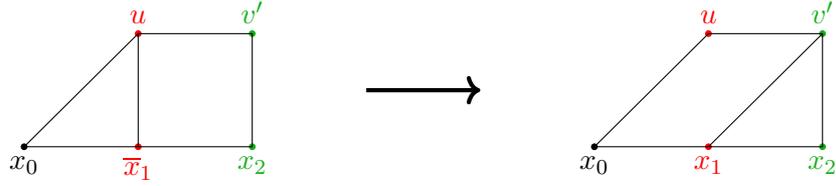
\begin{figure}[h!]
\centering
\begin{tikzpicture} [scale=1.5]
\draw[fill] (0,0) circle [radius=0.025]; 
\draw[fill,red] (1,0) circle [radius=0.025]; 
\draw[fill,red] (1,1) circle [radius=0.025]; 
\draw[fill,green] (2,0) circle [radius=0.025];
\draw[fill,green] (2,1) circle [radius=0.025];
\draw (0,0)--(1,1)--(2,1)--(2,0)--(0,0);
\draw (1,0)--(1,1);
\node[below] at (0,0) {$x_0$};
\node[below,red] at (1,0) {$\overline{x}_1$};
\node[above,red] at (1,1) {$u$};
\node[below,green] at (2,0) {$x_2$};
\node[above,green] at (2,1) {$v'$};

\draw[fill] (0+5,0) circle [radius=0.025]; 
\draw[fill,red] (1+5,0) circle [radius=0.025]; 
\draw[fill,red] (1+5,1) circle [radius=0.025]; 
\draw[fill,green] (2+5,0) circle [radius=0.025];
\draw[fill,green] (2+5,1) circle [radius=0.025];
\draw (0+5,0)--(1+5,1)--(2+5,1)--(2+5,0)--(0+5,0);
\draw (1+5,0)--(2+5,1);
\node[below] at (0+5,0) {$x_0$};
\node[below,red] at (1+5,0) {$x_1$};
\node[above,red] at (1+5,1) {$u$};
\node[below,green] at (2+5,0) {$x_2$};
\node[above,green] at (2+5,1) {$v'$};
\draw[ultra thick, ->] (3,0.5)--(4,0.5);

\end{tikzpicture}
\caption{Situation assumed that $u\simeq \overline{x}_1$ (for the sake of contradiction of Claim 3)}
\label{fig:claim_uv}
\end{figure}

Now we know $\square x_0u v' x_1$ is a quadrilateral without diagonal edges, and the same holds true for $\square x_0u vx_1$. However, Lemma \ref{lem:CP2} asserts that $v'=v$. This is impossible since $v'\sim x_2$ but $v\not\sim x_2$ (from Claim 2). Therefore Claim 3 must be true: $u\not\sim \overline{x}_1$.

Thus far, $u\not\simeq \overline{x}_1$ from Claims 3 and 1. On the other hand, $\overline{v}=\overline{T}_1(u)\in B_1(\overline{x}_1)$ as the codomain of $\overline{T}_1$, so $u\not=\overline{v}$. Corollary \ref{cor:moving} then implies that $\overline{v}=\overline{T}_1(u)$ must lie in $S_2(x_0)$ and that $u\sim \overline{v}$, as stated in Claim 4.

Next, we prove Claim 5 indirectly: assume that $\overline{v}\simeq x_2$. The fact that from $\overline{v}\not= x_2$ Claim 1 means $\overline{v}\sim x_2$.  Under this assumption, we have the following information as illustrated in the left diagram of Figure \ref{fig:claim_vw}: $u\in S_1(x_0)$, $u\not\sim \overline{x}_1$ and $u\not\sim x_2$ (from Claim 2) and $\overline{v}\in S_2(x_0)$ and $u\sim \overline{v}\sim x_2$ (from Claim 4 and assumption) and $\overline{v}\sim \overline{x}_1$ (as $\overline{v}$ is in the codomain of $\overline{T}_1$). Lemma \ref{lem:CP_swap2} then implies that $x_1,u,\overline{v},x_2$ forms a quadrilateral (compare Figure \ref{fig:CP_swap_lemma} to Figure \ref{fig:claim_vw}, where $v$ is replaced by $v'$). This contradicts to the fact that $u\not\sim x_1$ in Claim 2. Therefore, Claim 5 must be true.

\begin{figure}[h!]
\centering
\begin{tikzpicture} [scale=1.5]
\draw[fill] (0,0) circle [radius=0.025]; 
\draw[fill,red] (1,0) circle [radius=0.025]; 
\draw[fill,red] (1,1) circle [radius=0.025]; 
\draw[fill,green] (2,0) circle [radius=0.025];
\draw[fill,green] (2,1) circle [radius=0.025];
\draw (0,0)--(1,1)--(2,1)--(2,0)--(0,0);
\draw (1,0)--(2,1);
\node[below] at (0,0) {$x_0$};
\node[below,red] at (1,0) {$\overline{x}_1$};
\node[above,red] at (1,1) {$u$};
\node[below,green] at (2,0) {$x_2$};
\node[above,green] at (2,1) {$\overline{v}$};

\draw[fill] (0+5,0) circle [radius=0.025]; 
\draw[fill,red] (1+5,0) circle [radius=0.025]; 
\draw[fill,red] (1+5,1) circle [radius=0.025]; 
\draw[fill,green] (2+5,0) circle [radius=0.025];
\draw[fill,green] (2+5,1) circle [radius=0.025];
\draw (0+5,0)--(1+5,1)--(2+5,1)--(2+5,0)--(0+5,0);
\draw (1+5,0)--(1+5,1);
\node[below] at (0+5,0) {$x_0$};
\node[below,red] at (1+5,0) {$x_1$};
\node[above,red] at (1+5,1) {$u$};
\node[below,green] at (2+5,0) {$x_2$};
\node[above,green] at (2+5,1) {$\overline{v}$};
\draw[ultra thick, ->] (3,0.5)--(4,0.5);

\end{tikzpicture}
\caption{Situation assumed that $\overline{v}\sim x_2$ (for the sake of contradiction of Claim 5)}
\label{fig:claim_vw}
\end{figure}
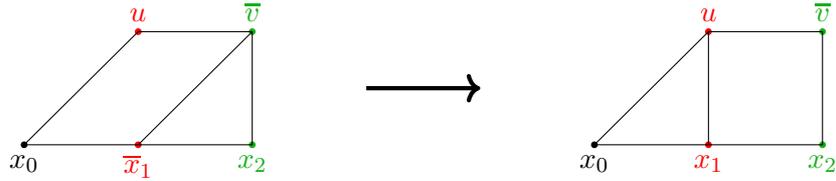

Thus far, $\overline{v}\not\simeq x_2$ from Claims 5 and 1. Since $\overline{v}\in S_2(x_0)$ from Claim 4, Corollary \ref{cor:moving} implies that $\overline{w}=\overline{T}_2(\overline{v})$ must lie in $S_3(x_0)$ and that $\overline{v} \sim \overline{w}$, which means Claim 6 is true. Note that Claim 7 is a combination of Claim 4 and 6, and it can be illustrated as in Figure \ref{fig:uvw_2}.

Now we are ready to prove the statement of the theorem. The fact that both $u$ and $x_2$ lie in $[x_0,w]$ can be seen from $x_0\sim u\sim v \sim w$ and $x_0\sim x_1 \sim x_2 \sim w$ (where $x_2 \sim w$ comes from the fact that $w=T_2(v) \in B_1(x_2)$ and $w\not=x_2$). We are left to prove that $d(u,x_2)=3$.

Obviously, $d(u,x_2)\le 3$, and $d(u,x_2)>1$ because $u\not=x_2$ (since $u\in S_1(x_0)$ but $x_2\in S_2(x_0)$) and $u\not\sim x_2$ from Claim 2. Assume for the sake of contradiction that $d(u,x_2)=2$. Then there is a vertex $b$ such that $u\sim b\sim x_2$. Note that $b$ is in either $S_1(x_0)$ or $S_2(x_0)$. The next paragraph argues that we can assume without loss of generality that $b\in S_2(x_0)$.

In case $b\in S_1(x_0)$, the geodesic extension of $x_0\sim b\sim x_2$ (see Remark \ref{rem:geo_ext}) allows us to consider good optimal transport maps $B_1(x_0)\stackrel{\tau_1}{\rightarrow}  B_1(b) \stackrel{\tau_2}{\rightarrow} B_1(x_2)$. Since $u\in S_1(x_0)$ and $u\sim b$ but $u\not\sim x_2$ (from Claim 2), Lemma \ref{lem:prop_T2}(c) implies that the vertex $\beta:=\tau_2(u)$ satisfies $\beta\in S_2(x_0)$ and $u\sim \beta\sim x_2$ (see Figure \ref{fig:beta_b} compared to Figure \ref{fig:T2_map}(c) with vertices $b,u,\beta$ replacing $x_1,v,w$, respectively). Hence we can assume without loss of generality that $b\in S_2(x_0)$; otherwise, we may consider $\beta$ instead $b$.

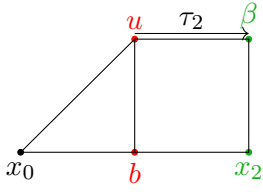
\begin{figure}[h!]
\centering
\begin{tikzpicture} [scale=1.5]
\draw[fill] (0,0) circle [radius=0.025]; 
\draw[fill,red] (1,0) circle [radius=0.025]; 
\draw[fill,red] (1,1) circle [radius=0.025]; 
\draw[fill,green] (2,0) circle [radius=0.025];
\draw[fill,green] (2,1) circle [radius=0.025];
\draw[fill,white] (3,0) circle [radius=0.025]; 
\draw (0,0)--(1,0)--(2,0)--(2,1)--(1,1)--(0,0);
\draw (1,0)--(1,1);
\draw[->] (1,1+.05)--(2,1+.05);
\node[below] at (0,0) {$x_0$};
\node[below,red] at (1,0) {$b$};
\node[above,red] at (1,1) {$u$};
\node[below,green] at (2,0) {$x_2$};
\node[above,green] at (2,1) {$\beta$};
\node[above] at (1.5,1) {$\tau_2$};
\end{tikzpicture}
\caption{If $u\sim b\sim x_2$ with $b\in S_1(x_0)$, then $u\sim \beta \sim x_2$ with $\beta\in S_2(x_0)$.}
\label{fig:beta_b}
\end{figure}

Now we are able to assume $b\in S_2(x_0)$ and $u\sim b \sim x_2$. Since $b\in S_2(x_0)$ and $b\sim x_2$, Lemma \ref{lem:XOR} implies that either $b\sim x_1$ or $b\sim \overline{x}_1$. If $b\sim x_1$, then $x_0,u,b,x_1$ form a quadrilateral with $u\not\sim x_1$ (from Claim 2). Since we also have a quadrilateral $\square x_0uvx_1$ (see Figure \ref{fig:uvw}(a)), so the uniqueness by Lemma \ref{lem:CP2} guarantees that $b=v$, which is impossible since $b\sim x_2$ but $v\not\sim x_2$ (from Claim 2). By similar arguments, if $b\sim \overline{x}_1$, then $x_0,u,b,\overline{x}_1$ form a quadrilateral with $u\not\sim\overline{x}_1$ (from Claim 3). By comparing this quadrilateral to $\square x_0u\overline{v} x_1$ (see Figure \ref{fig:uvw}(b)), the uniqueness by Lemma \ref{lem:CP2} implies $b=\overline{v}$, which is also impossible since $b\sim x_2$ but $\overline{v}\not\sim x_2$ (from Claim 5).

In conclusion,  $d(u,x_2)=3$ as desired.
\end{proof}

Another important theorem that follows from Theorem \ref{thm:uvw} is about the existence of antipoles with respect to intervals of length three.
\begin{theorem} \label{thm:dist3_sc}
Let $x_0$ be a pole and $x_3\in S_3(x_0)$. Then for each vertex $y\in [x_0,x_3]$ such that $y\in S_2(x_0)$, there exists another vertex $u\in [x_0,x_3]$ such that $d(u,y)=3$.
\end{theorem}

\begin{proof}[Proof of Theorem \ref{thm:dist3_sc}]
For $y\in [x_0,x_3]$ such that $y\in S_2(x_0)$, we label $y$ as $x_2$. Choose an arbitrary vertex $x_1$ such that $x_0\sim x_1 \sim x_2$, and consider good optimal transport maps $B_1(x_0) \stackrel{T_1}{\rightarrow} B_1(x_1) \stackrel{T_2}{\rightarrow}  B_1(x_2)$. Let $w:=x_3 \in S_3(x_0)$, and let vertices $u,v$ be the preimages $u\stackrel{T_1^{-1}}{\mapsfrom} v \stackrel{T_2^{-1}}{\mapsfrom} w$. Lemma \ref{lem:prop_T2}(a) implies that $v\in S_2(x_0)$. We also know $u\in B_1(x_0)$ since it lies in the domain of $T_1$. However, observe that $u\not=x_0$; otherwise, $v=T_1(u)=T_1(x_0)=\overline{x}_1\in S_1(x_0)$ by Lemma \ref{lem:CP1}. Therefore, $u\in S_1(x_0)$, $v\in S_2(x_0)$, $w\in S_3(x_0)$, and we can apply Theorem \ref{thm:uvw} to conclude that the vertex $u$ satisfies $u\in [x_0,x_3]$ and $d(u,y)=3$.
\end{proof}

\begin{corollary} \label{cor:bm3}
Let $G = (V,E)$ be a $D$-regular Bonnet-Myers sharp graph with diameter $L=3$, then $G$ is self-centered. 
\end{corollary}

\begin{proof}[Proof of Corollary \ref{cor:bm3}]
Let $x_0$ and $x_3$ be antipoles of each other. We know from \cite[Theorem 5.5]{rigidity} that $[x_0,x_3]=V$, so any vertex $y$ different from $x_0$ and $x_3$ must lie in $S_2(x_0)$ or $S_2(x_3)$. In either case, we can apply Theorem \ref{thm:dist3_sc} (since both $x_0$ and $x_3$ are poles) and conclude that there exists a vertex $u$ such that $d(u,y)=3$.
\end{proof}

Finally, our main result (Theorem \ref{thm:bm_diam3}) is an immediate consequence of Corollary \ref{cor:bm3} with the classification given in Theorem \ref{thm:bm_sc} in the introduction.

\end{document}